\documentclass[a4paper,10pt]{amsart}
\usepackage[czech,english]{babel}
\usepackage{amsmath,amssymb,amsthm,amscd}
\usepackage[all]{xy}

\newcommand{\ModR}{\hbox{{\rm Mod-}}R}

\newcommand{\Add}{\mathrm{Add}}

\DeclareMathOperator{\Hom}{Hom}
\DeclareMathOperator{\Dom}{Dom}
\DeclareMathOperator{\End}{End}
\DeclareMathOperator{\Ext}{Ext}

\DeclareMathOperator{\Ker}{Ker}
\DeclareMathOperator{\Img}{Im}

\DeclareMathOperator{\Soc}{Soc}

\DeclareMathOperator{\cf}{cf}

\usepackage[usenames]{color}

\theoremstyle{plain}
\newtheorem{thm}{Theorem}[section]
\newtheorem{prop}[thm]{Proposition}
\newtheorem{lem}[thm]{Lemma}
\newtheorem*{conj}{Conjecture}
\newtheorem{cor}[thm]{Corollary}
\theoremstyle{definition}

\newtheorem{exm}[thm]{Example}

\theoremstyle{remark}
\newtheorem*{rema}{Remark}

\begin{document}
\title{Enochs' conjecture for small precovering classes of modules}

\author{\textsc{Jan \v Saroch}}
\address{Charles University, Faculty of Mathematics and Physics, Department of Algebra \\ 
Sokolovsk\'{a} 83, 186 75 Praha~8, Czech Republic}
\email{saroch@karlin.mff.cuni.cz}


\thanks{Research supported by GA\v CR 20-13778S}

\date{\today}

\begin{abstract} Enochs' conjecture asserts that each covering class of modules (over any fixed ring) has to be closed under direct limits. Although various special cases of the conjecture have been verified, the conjecture remains open in its full generality. In this short paper, we prove the validity of the conjecture for small precovering classes, i.e.\ the classes of the form $\Add(M)$ where $M$ is any module, under a mild additional set-theoretic assumption which ensures that there are enough non-reflecting stationary sets. We even show that $M$ has a perfect decomposition if $\Add(M)$ is a covering class. Finally, the additional set-theoretic assumption is shown to be redundant if there exists an $n<\omega$ such that $M$ decomposes into a direct sum of $\aleph_n$-generated modules.
\end{abstract} 

\maketitle
\vspace{4ex}

\section*{Introduction}
\label{sec:intro}

Precovers (and their dual counterpart preenvelopes) represent the basic tools in the approximation theory of modules. A class $\mathcal C\subseteq\ModR$ is called \emph{precovering} if each $M\in\ModR$ possesses a \emph{$\mathcal C$-precover}, i.e.\ a homomorphism $f:C\to M$ where $C\in \mathcal C$ and such that $\Hom_R(C^\prime,f)$ is surjective for each $C^\prime\in\mathcal C$. Depending on the class $\mathcal C$, the $\mathcal C$-precovers are often epimorphisms but, in general, they need not be. Rada and Saorín in \cite{RS0} noticed that a class $\mathcal C$ is precovering if and only if its closure under direct summands is precovering. In other words, when studying precovering classes, we can concentrate on the ones closed under direct summands. It is an easy exercise to show that, in this case, they are closed under direct sums, too. As a~consequence, we can see that, unless $\mathcal C$ contains only the zero module, $\mathcal C$-precovers are not unique by any means.

Some precovering classes, though, provide us with minimal versions of precovers called covers. A $\mathcal C$-precover $f:C\to M$ satisfying that each $g\in\End_R(C)$ such that $fg = f$ is an automorphism of $C$ is called a \emph{$\mathcal C$-cover} of $M$. A class $\mathcal C\subseteq\ModR$ is \emph{covering} if all modules have $\mathcal C$-covers. The $\mathcal C$-covers, if exist, are unique up to isomorphism. However, they still need not be surjective. And also, even if $\mathcal C$-covers exist, they need not be functorial.

Precovering and preenveloping classes play an important role in relative homological algebra. They allow us to define (relative) resolutions and coresolutions and also to introduce meaningful notions of (relative homological) dimensions in some cases. Their minimal versions, covers and envelopes, are used, for instance, in the definition of Bass' invariants over commutative noetherian rings, or Xu's dual Bass invariants over Gorenstein rings.

The basic examples of precovering classes are $\mathcal P_0$ and $\mathcal F_0$, i.e.\ the class of all projective and flat modules, respectively. In a more general fashion, given a set $\mathcal S\subseteq\ModR$, the class $\Add(\mathcal S)$ of all direct summands of arbitrary direct sums of modules from $\mathcal S$ is precovering by the Quillen's small object argument, cf.\ \cite[Corollary~3.7]{RS0}. Of course, in this case $\Add(\mathcal S) = \Add(\{\bigoplus_{M\in\mathcal S} M\})$ where we usually drop the braces in the latter expression and write simply $\Add(\bigoplus_{M\in\mathcal S}M)$.

Precovering classes of the form $\Add(M)$ where $M\in \ModR$ (or equivalently $\Add(\mathcal S)$ where $\mathcal S$ is a \emph{subset} of $\ModR$) will be called \emph{small precovering classes} in this paper. Given an infinite cardinal $\kappa$, it follows immediately from the well-known Walker's lemma that each module in $\Add(\mathcal S)$ is a direct sum of $\kappa$-presented modules provided that all modules in $\mathcal S$ have this property.

The rings where $\mathcal P_0 = \Add(R_R)$ is covering are called \emph{right perfect}. The class $\mathcal F_0$ is covering\footnote{Note, however, that $\mathcal F_0$ is not a small precovering class unless it coincides with $\mathcal P_0$.} over any ring $R$ as shown in the famous paper \cite{BBE0} by Bican, El Bashir and Enochs. Among other things, Enochs proved that a precovering class closed under direct limits is covering, cf.\ \cite[Theorem~5.31]{GT0}. The question whether the converse implication holds as well is known as the Enochs' conjecture.

\begin{conj}\label{conj:Enochs} (Enochs) Every covering class of modules is closed under direct limits.
\end{conj}

This conjecture has been verified for various special types of classes: from the Bass' famous Theorem~P, \cite{Ba0}, we know that it holds for $\mathcal P_0$ (i.e.\ the existence of projective covers implies that $\mathcal P_0$ is closed under direct limits); recently, Bazzoni and Le Gross in \cite{BL0} verified the conjecture for the class $\mathcal P_1$ of modules of projective dimension at most $1$ over a commutative semihereditary ring; it is shown in \cite[Theorem~5.2]{AST} that Enochs' conjecture holds for each class $\mathcal A$ which appears in a~cotorsion pair $(\mathcal A,\mathcal B)$ with $\mathcal B$ closed under direct limits. A promising contramodule-based approach presented in \cite{BP} gives a positive solution for some special cases of small precovering classes. However, the general case of the conjecture still remains open.

\smallskip

In this short paper, we verify the Enochs' conjecture, under an additional incompactness set-theoretic assumption, for all small precovering classes, see Theorem~\ref{t:AddEnochs}. In fact, we will even show a~little bit more: the module $M$ has to have a~perfect decomposition provided that $\Add(M)$ is covering. Also the additional set-theoretic assumption is not needed if there exists an $n<\omega$ such that $M$ is a~direct sum of $\aleph_n$-generated modules, see Corollary~\ref{c:broader}.

\section{Preliminaries}
\label{sec:prelim}

Throughout this paper, $R$ denotes an associative unital ring and $\ModR$ the category of all (right $R$-)modules and homomorphisms between them. We note, however, that our results generalize in a straightforward way to the category of unitary modules over a ring with enough idempotents and potentially also to arbitrary finitely accessible additive category.

Unless otherwise stated, we work in the classical realm of ZFC, i.e.\ Zermelo--Fraenkel set theory with the axiom of choice. We also use the usual definitions of ordinal numbers, cardinal numbers, cofinality and stationary subsets (in particular, $\omega$ denotes the least infinite ordinal number). The reader unfamiliar with some of these notions is encouraged to consult the monograph \cite{J0} or \cite{EM0}.

Given an infinite regular cardinal $\lambda$, we say that a partially ordered set $(I,\leq)$ is \emph{$\lambda$-directed} provided that each subset $J\subseteq I$ with $|J|<\lambda$ has an upper bound in $(I,\leq)$. A system $\mathcal M = (M_i,f_{ji}:M_i\to M_j \mid i\leq j\in I)$ of modules and homomorphisms between them indexed by a $\lambda$-directed poset $(I,\leq)$ is called \emph{$\lambda$-directed} as well. If $\lambda = \aleph_0$, we often omit $\lambda$.

For a $\lambda$-directed system $\mathcal M$ as above, we have the canonical short exact sequence $0\to K\to \bigoplus_{i\in I}M_i \overset{\pi}{\to} \varinjlim\mathcal M \to 0$ where the epimorphism $\pi$ is \emph{$\lambda$-pure}, i.e.\ it has the property that $\Hom_R(F,\pi)$ is surjective whenever $F$ is a $<\lambda$-presented module.

Recall that a \emph{filtration} of a module $M$ is a sequence $\mathfrak F = (M_\alpha\mid \alpha\leq \sigma)$ of modules indexed by an ordinal $\sigma$ such that $M_0 = 0$, $M_\sigma = M$, $M_\alpha\subseteq M_\beta$ for each $\alpha<\beta\leq\sigma$ and $M_\alpha = \bigcup_{\beta<\alpha} M_\beta$ for each limit ordinal $\alpha\leq\sigma$. Moreover, we say that $\mathfrak F$ is \emph{splitting} if $M_\alpha$ is a direct summand in $M_\sigma$ for each $\alpha<\sigma$. Notice that, if $\mathfrak F$ is splitting, then $M\cong \bigoplus_{\alpha<\sigma} M_{\alpha+1}/M_\alpha$.

We say that a module $M$ \emph{has a perfect decomposition} if every local direct summand in a module from $\Add(M)$ is a direct summand. We call a submodule $\bigoplus_{i\in I} N_i$ in $N$ a \emph{local direct summand}\footnote{See also the notion of a quasi-split monomorphism in \cite{BPS0}.} in $N$ if the subsum $\bigoplus_{i\in J} N_i$ is a direct summand in~$N$ for each finite $J\subseteq I$. This notion was studied, for instance, in \cite{GG0} and~\cite{AS0}. In particular, it follows from \cite[Corollary~2.3]{GG0} that a module with perfect decomposition has a decomposition in modules with local endomorphism ring; and by \cite[Theorem~1.4]{AS0}, we know that $\Add(M)$ is closed under direct limits if $M$ has a~perfect decomposition. Thus, for any $M\in\ModR$, we have
\[M\mbox{ has a perf.\ decom. } \Longrightarrow\Add(M)\mbox{ is closed under }\varinjlim \Longrightarrow\Add(M)\mbox{ is covering}\]
and we are going to show that, under a mild additional set-theoretic assumption, the converses of these implications hold as well.

\begin{exm} \label{e:ex} 
\begin{enumerate}
	\item Since any local direct summand is a pure submodule, an $M\in\ModR$ has a perfect decomposition provided that it is $\Sigma$-pure-split,\ i.e.\ every pure submodule of a direct sum of copies of $M$ is a direct summand; e.g.\ $R_R$ is $\Sigma$-pure-split if and only if $R$ is right perfect, if and only if $R_R$ has a perfect decomposition. Let us note, however, that there exist modules with perfect decomposition which are not $\Sigma$-pure-split.

  \item Consider the boolean ring $R = \mathcal P(\omega)$ where, as usual, the addition is the symmetric difference and the multiplication is the intersection. Then $\bigoplus_{a\in\omega} \{a\}R = \Soc R$ is a local direct summand in the regular module $R$ which is not a direct summand in~$R$. So $R$ does not have a perfect decomposition. Of course, this is not surprising: otherwise $\Add(R) = \mathcal P_0$ would be closed under direct limits, and $R$ would be semisimple artinian.
\end{enumerate}
\end{exm}

\section{The main theorem}
\label{sec:main}

We start with a simple, yet very useful proposition which allows us to capitalize on the covering assumption. It is due to Bazzoni, Positselski and Šťovíček. Recall that a~morphism $m:K\to M$ is called \emph{locally split} provided that \[(\forall x\in K)(\exists h\in\Hom_R(M,K))\, h(m(x))=x.\] In particular, if $K = \bigoplus_{i\in I}N_i$ is a local direct summand in a module $N$, then the inclusion $K\hookrightarrow N$ is locally split.

\begin{prop} \label{p:localsplit} Let $\mathcal C\subseteq\ModR$ and $f\in\Hom_R(N, L)$ be a surjective $\mathcal C$-precover with a~locally split kernel. If $L$ has a $\mathcal C$-cover, then $L\in\mathcal C$ and the epimorphism $f$ splits.
\end{prop}

\begin{proof} See \cite[Corollary~5.3]{BPS0} where the authors prove a general version for Ab5 categories.
\end{proof}

Apart from Proposition~\ref{p:localsplit}, we shall use the following additional set-theoretic

\medskip

\noindent\textbf{Assumption $(*)$.} \textit{There is a proper class of cardinals $\kappa$ such that each stationary set $E^\prime\subseteq\kappa^+$ has a non-reflecting stationary subset $E$.}

\medskip

Recall that a subset $E\subseteq\kappa^+$ is \emph{non-reflecting} if $E\cap\alpha$ is non-stationary in $\alpha$ whenever $\alpha<\kappa^+$ is a limit ordinal. The principle $(*)$ is consistent with ZFC since it holds, for instance, in the constructible universe or, more generally, in absence of~$0^\#$. It brings some incompactness to the universe of sets: simply put, it postulates the existence of many `large' (i.e.\ stationary) sets which are locally `small' (i.e.\ non-stationary).

The construction in the proof of our main theorem below is inspired by the one from \cite[Theorem~VII.2.3]{EM0}. We do not know whether Theorem~\ref{t:AddEnochs} can be proved in ZFC alone; it is possible for some special cases of $M$ as showed in Corollary~\ref{c:broader}, however, the general case is open as is the Enochs' conjecture~itself.

\begin{thm} \label{t:AddEnochs} Assume that $(*)$ holds true. Let $M\in\ModR$ be such that each module in $\varinjlim\Add(M)$ has an~$\Add(M)$-cover. Then $M$ has a perfect decomposition. In particular, $\Add(M)$ is closed under direct limits.
\end{thm}

\begin{proof} Let $\mu$ be an infinite cardinal such that $M$ is a direct sum of $\mu$-presented modules. Assume that $M$ does not have a perfect decomposition. Then there exists a local direct summand $K = \bigoplus_{i\in I} N_i$ in a module $N\in\Add(M)$ which is not a~direct summand there. By the Walker's lemma, we can w.l.o.g.\ assume that each $N_i$ is $\mu$-presented. Aiming for minimality, we can also assume that each submodule $\bigoplus_{i\in J} N_i$ where $|J|<|I|$ is an actual direct summand in $N$. If $|I|>\mu$, then the inclusion $\nu:K\hookrightarrow N$ is the directed union of a $\mu^+$-directed system consisting of split inclusions into $N$. It is thus trivially a locally split morphism. Moreover, since $M$ is a direct sum of $\mu$-presented modules, any homomorphism from $M$ into $N/K$ factorizes through the $\mu^+$-pure canonical projection $\pi: N\to N/K$. So $\pi$ is an $\Add(M)$-precover and Proposition~\ref{p:localsplit} implies that $\pi$ splits, a~contradiction.

Thus $\aleph_0\leq |I| \leq \mu$. If $|I|$ is singular, we write $I$ as a disjoint union $\bigcup_{j\in J} I_j$ of subsets $I_j$ where $|J|=\cf(|I|)$ is regular and $|I_j|<|I|$ for each $j\in J$; then $\bigoplus_{j\in J} \bigl(\bigoplus_{i\in I_j} N_i\bigr)$ equals to~$K$, whence it is not a direct summand in $N$, however, it is still a local direct summand by the minimality of $I$. Hence we may assume without loss of generality that $I = \lambda \leq \mu$ is an infinite regular cardinal and $K$ is $\mu$-presented. Since $N$ is a direct sum of $\mu$-presented modules (again, by the Walker's lemma), we can also assume that $N$ itself is $\mu$-presented. Using $(*)$, we pick a suitable $\kappa\geq\mu$ and fix a non-reflecting stationary set $E\subseteq \kappa^+$ consisting of ordinal numbers with cofinality $\lambda$. (In $(*)$, take the stationary set $E^\prime = \{\alpha<\kappa^+\mid \cf(\alpha) = \lambda\}$.)

We apply a certain homogenization procedure based on the well-known Eilenberg's trick, i.e.\ on the fact that, if $A$ is a direct summand in $B$, then $A\oplus B^{(\nu)}\cong B^{(\nu)}$ whenever $\nu$ is infinite. Put $H = (K\oplus N)^{(\kappa)}$. This is a $\kappa$-presented module from $\Add(M)$. We define a splitting filtration $\mathfrak F = (M_\alpha \mid \alpha\leq\lambda)$ recursively by putting $M_{\alpha+1} = M_\alpha \oplus N_\alpha \oplus H$ and taking unions in the limit steps. Then $M_\lambda = K\oplus H^{(\lambda)} \subseteq N\oplus H^{(\lambda)}$. Notice that we have $N\oplus H^{(\lambda)}\cong H\cong M_{\alpha+1}/M_\alpha$ for each $\alpha<\lambda$, and $M_\alpha\cong H$ for each nonzero $\alpha\leq\lambda$, as a~consequence of the Eilenberg's trick. Also observe that \[M_\alpha\mbox{ is a direct summand in }N\oplus H^{(\lambda)}\mbox{ for each }\alpha<\lambda, \eqno{(\dagger)}\] whilst $M_\lambda$ does not split in $N\oplus H^{(\lambda)}$.

We are going to extend the filtration $\mathfrak F$ to a filtration $\mathfrak H = (M_\alpha \mid \alpha\leq\kappa^+)$. While defining it, we ensure that:

\begin{enumerate}
	\item[(i)] for each $0<\alpha<\kappa^+$, $M_\alpha\cong H\in\Add(M)$;
	\item[(ii)] for each $\beta<\alpha<\kappa^+$ with $\beta\notin E$, $M_\alpha = M_\beta \oplus G$ for some $G\subseteq M_\alpha$ isomorphic to $H$;
	\item[(iii)] if $\beta\in E$, then $M_\beta$ does not split in $M_{\beta+1}$.
\end{enumerate}

These conditions are clearly satisfied for the piece $\mathfrak F$ of $\mathfrak H$ we have constructed so far. Assume now that $\lambda<\alpha<\kappa^+$ and that $M_\beta$ is already defined for every $\beta<\alpha$. We distinguish the following three cases.

\begin{enumerate}
	\item $\alpha$ is limit: then we know that $E\cap \alpha$ is not stationary, thus we can find a~closed and unbounded subset $S$ of $\alpha$ such that $S\cap E = \varnothing$. We have to define $M_\alpha = \bigcup_{\gamma<\alpha} M_{\gamma} = \bigcup_{\gamma\in S} M_{\gamma}$. Since $\mathfrak G = (M_\gamma \mid \gamma\in S)$ is a~splitting filtration of $M_\alpha$ where the consecutive factors are isomorphic to $H$, we conclude that $M_\alpha \cong \bigoplus _{\gamma\in S} H \cong H$. This gives us (i). The condition (ii) then follows easily: given any $\beta<\alpha$ with $\beta\notin E$, we first find $\gamma \in S$ such that $\beta<\gamma$; then we use (ii) for $\beta<\gamma$ and the fact that $M_\gamma$ is a direct summand in $M_\alpha$ from the splitting filtration $\mathfrak G$.
	\item $\alpha = \beta + 1$ for $\beta\in E$: then we have a splitting filtration $\mathfrak G = (M_\gamma \mid \gamma\in S)$ of~$M_\beta$ from the previous limit step. Since $\beta\in E$, we have $\cf(\beta) = \lambda$, and so we can assume w.l.o.g.\ that $S$ has order type $\lambda$, i.e.\ we can enumerate $\mathfrak G = (M_{\gamma_\delta} \mid \delta<\lambda)$. By our construction, there exists an isomorphism $\iota:M_\lambda \to M_\beta$ such that $\iota\restriction M_\delta$ is an isomorphism onto $M_{\gamma_\delta}$ for each $\delta<\lambda$. Recall that $M_\lambda = K\oplus H^{(\lambda)}$. We define $M_\alpha$ using the pushout of $\iota$ and the inclusion $K\oplus H^{(\lambda)}\subseteq N\oplus H^{(\lambda)}$ which is non-split since $K$ is not a direct summand in $N$. Thus $M_\beta$ is not a direct summand in $M_\alpha\cong N\oplus H^{(\lambda)}\cong H$ either, and we have (i) and (iii) checked. The condition (ii) follows immediately from the property of $\iota$ and $(\dagger)$.
	\item $\alpha = \beta +1$ for $\beta\notin E$: in this case, we simply put $M_\alpha = M_\beta \oplus H$ and immediately check that the conditions (i)--(iii) are, indeed, satisfied for $\alpha$.
\end{enumerate}

Now, $Z = M_{\kappa^+}$ is the directed union of the $\kappa^+$-directed system $(M_\alpha \mid \alpha<\kappa^+)$ consisting of modules in $\Add(M)$. Since $M$ is a~direct sum of $\kappa$-presented modules (even $\kappa$-generated would suffice here), every $g\in\Hom_R(M,Z)$ factorizes through the canonical epimorphism $z:\bigoplus_{\alpha<\kappa^+}M_\alpha \to Z$ yielding that $z$ is an $\Add(M)$-precover. It is well-known that $\Ker(z)$ is locally split (cf.\ \cite[Lemma~2.1]{GG0}), whence we deduce that $z$ is a~split epimorphism and $Z\in\Add(M)$ by Proposition~\ref{p:localsplit}. By Walker's lemma, we know that $Z$ is a direct sum of $\kappa$-presented modules. This gives us a~filtration $\mathfrak H^\prime = (M_\alpha^\prime \mid \alpha\leq\kappa^+)$ of $Z$ such that $M_\alpha^\prime$ is a $\kappa$-presented direct summand in $Z$ for each $\alpha<\kappa^+$. It follows that the set $T = \{\alpha<\kappa^+ \mid M_\alpha = M_\alpha^\prime\}$ is closed and unbounded, whence we can pick a~$\beta \in T\cap E$. Then $M_\beta$ splits in $Z$, and so it splits in $M_{\beta+1}$, too, in contradiction with (iii).

We have proved that $M$ has a perfect decomposition. Finally, it follows from \cite[Theorem~1.4]{AS0} that $\Add(M)$ is closed under direct limits.
\end{proof}

Let us shortly sum up what we have done in the proof above. Assuming that $\Add(M)$ is not closed under direct limits, we know (e.g.\ from \cite[Theorem~1.4]{AS0}) that $M$ does not have a perfect decomposition. Hence there exists a local direct summand $K$ in a module $N\in\Add(M)$ which is not a direct summand. Using the non-split inclusion $K\subseteq N$ and our assumption $(*)$, we are able to build arbitrarily large modules $Z\in\varinjlim\Add(M)$ which locally look like elements of $\Add(M)$, i.e.\ each submodule $Y$ of~$Z$ with $|Y|<|Z|$ is contained in a submodule of $Z$ belonging to $\Add(M)$, but which do not decompose as a direct sum of modules of cardinality strictly smaller than $|Z|$. On the other hand, Proposition~\ref{p:localsplit} essentially implies that this is impossible if $|Z|>|M|$ and $Z$ has an $\Add(M)$-cover.

\begin{rema} If $M$ is a direct sum of countably presented modules, the proof of Theorem~\ref{t:AddEnochs} does not need the extra assumption $(*)$. The point is that if $\lambda = \mu (= \aleph_0)$ happens in this case, then $\kappa = \aleph_0$ works since the set $E\subseteq \aleph_1$ consisting of all limit ordinals is stationary and non-reflecting.

More generally, if it happens in the second paragraph of the proof that $\mu = \lambda$, then we can put $\kappa = \lambda$ and we do not need to use $(*)$ since already the stationary set $E^\prime = \{\alpha<\kappa^+\mid \cf(\alpha) = \kappa\}$ is non-reflecting.
\end{rema}

The following corollary puts Theorem~\ref{t:AddEnochs} into a broader picture of established theory. Moreover, using a technical lemma which we defer to the appendix, we can also present an important instance, generalizing the one mentioned in the first paragraph of the remark above, where the assumption $(*)$ is redundant.

Let us recall that a ring $S$ is \emph{semiregular} if the quotient of $S$ modulo the Jacobson radical is von Neumann regular and idempotents lift modulo the Jacobson radical. An in-depth study of these rings can be found in \cite{N0}. Notice for instance that, if $R$ is the boolean ring from Example~\ref{e:ex}, then $\End_R(R^{(\omega)})$ is not semiregular, cf.\ \cite[Theorem~3.9]{N0}.

\begin{cor}\label{c:broader} Let $M\in\ModR$ and suppose that either:
\begin{enumerate}
	\item[(a)] the assumption $(*)$ holds true, or
	\item[(b)] there exists $n<\omega$ such that $M$ is a direct sum of $\aleph_n$-generated modules.
\end{enumerate}

Then the following conditions are equivalent:
\begin{enumerate}
	\item $M$ has a perfect decomposition.
	\item $\Add(M)$ is closed under direct limits.
	\item $\Add(M)$ is a covering class of modules.
  \item Each module from the class $\varinjlim\Add(M)$ has an $\Add(M)$-cover.
  \item For each cardinal $\kappa$, the ring $\End_R(M^{(\kappa)})$ is semiregular.
\end{enumerate}
\end{cor}

\begin{proof} Concerning the previously known implications that hold without assuming either of the two assumptions (a) and (b), $(1)\Longrightarrow (2)$ follows by \cite[Theorem~1.4]{AS0}. The implication $(2)\Longrightarrow (3)$ is due to Enochs and can be found in \cite[Theorem~5.31]{GT0}. The implication $(3)\Longrightarrow (4)$ is trivial and $(1)\Longrightarrow (5)$ follows from the characterization in \cite[Theorem~1.1]{AS0}. On the other hand, $(5)\Longrightarrow (3)$ follows from \cite[Corollary~2.3]{N0} and \cite[Proposition~4.1]{A0}.

Finally, the implication $(4)\Longrightarrow (1)$ follows by Theorem~\ref{t:AddEnochs} if (a) holds true. If (b) is satisfied and $(1)$ does not hold, there exists $N\in\Add(M)$ which satisfies the assumption of Lemma~\ref{l:ZFCind}, and we use this lemma to obtain a non-split epimorphism $\pi:N^\prime \to L$ whose properties (i)--(iii) imply that $\pi$ is an $\Add(M)$-precover of $L\in\varinjlim\Add(M)$ with $\Ker(\pi)\hookrightarrow N^\prime$ locally split. It follows from Proposition~\ref{p:localsplit} that $L$ cannot have an $\Add(M)$-cover.
\end{proof}

Let us remark that, by \cite[Theorem~1.1]{AS0} and \cite[Corollary~2.3]{N0}, the condition~$(1)$ from Corollary~\ref{c:broader} is equivalent (in ZFC) with the conjunction of the condition~$(5)$ with ``$M$ has a decomposition in modules with local endomorphism ring''. It might be of interest, from the point of view of the theory of decomposition of modules, that $(5)$ alone actually suffices to ensure that $M$ has a decomposition in modules with local endomorphism ring assuming (a) or (b). Also, as already pointed out prior to Theorem~\ref{t:AddEnochs}, we have no example showing that one of the two conditions (a), (b) is actually necessary to prove Corollary~\ref{c:broader}, so it may turn out that $(5)$ yields that $M$ has a decomposition in modules with local endomorphism ring unconditionally.

If this is the case, a proof of this fact would have to use a different technique. Note the following limitation of our approach: the assumption that $\Add(M)$ is covering (which is implied by $(5)$) is used here to show that, if a module $Z$ is the directed union of a $\kappa^+$-directed system of submodules belonging to $\Add(M)$ and $M$ is $\kappa$-generated, then $Z\in\Add(M)$ and the canonical $\kappa^+$-pure epimorphism from the direct sum of submodules onto their directed union $Z$ splits; from this consequent assertion, we are able to deduce that $\Add(M) = \varinjlim\Add(M)$ if (a) or (b) is satisfied. This will not work if a strongly compact cardinal $\kappa >|R|$ is present (which contradicts the assumption $(*)$): the consequent assertion holds true for this~$\kappa$ and $M = R^{(\kappa)}$, as shown in \cite[Theorem~3.3]{ST}, but the class $\Add(M)$ of projective modules is not closed under direct limits unless $R$ is right perfect.

\appendix
\section{Lemma allowing to drop the assumption $(*)$ in some cases}
\label{sec:appendix}

The lemma presented below is needed in Corollary~\ref{c:broader} to show that the assumption $(*)$ in our main result is redundant provided that there exists $n<\omega$ such that $M$ is a direct sum of $\aleph_n$-generated modules. The tree-like construction in the proof of the lemma goes back to \cite{ES}.

In what follows, we say that a local direct summand $K = \bigoplus_{i\in I} N_i$ of a module $N$ is \emph{irredundant} if $K$ is not a direct summand in $N$ and $\bigoplus_{i\in J} N_i$ is a direct summand in $N$ for each $J\subset I$ with $|J|<|I|$.

\begin{lem}\label{l:ZFCind} Assume that $N$ has a local direct summand which is not a direct summand and let $n<\omega$ be arbitrary. There exists a non-split short exact sequence $0 \to G \to N^\prime \overset{\pi}{\to} L \to 0$ such that:
\begin{enumerate}
\item[(i)] $N^\prime$ is a direct sum of copies of $N$ and $L\in\varinjlim\Add(N)$,
\item[(ii)] $\Hom_R(M,\pi)$ is surjective whenever $M$ is $\aleph_n$-generated, and
\item[(iii)] $G\to N^\prime$ is locally split.
\end{enumerate}
\end{lem}

\begin{proof} Let $K = \bigoplus_{i\in I} N_i$ be a local direct summand in $N$ which is not a direct summand. We may suppose that it is irredundant. Then $|I|$ is infinite, and if $|I|$ is singular, we write $I$ as a disjoint union $I = \bigcup_{j\in J} I_j$ where $|J|=\cf(|I|)$ is regular and $|I_j|<|I|$ for each $j\in J$. As in the proof of Theorem~\ref{t:AddEnochs}, we get that $K = \bigoplus_{j\in J} (\bigoplus_{i\in I_j} N_i)$ and this (decomposition of $K$) is still an irredundant local direct summand in $N$. As a result, we can assume without loss of generality that $I = \lambda$ where $\lambda$ is an infinite regular cardinal.

Put $C = N/K$ and let $f:N\to C$ be the (non-split) canonical projection. Take a cardinal $\kappa$ such that $|R|+|K|\leq\kappa = \kappa^{<\lambda}$ and $\kappa^\lambda = 2^\kappa$. Let $T = {}^\lambda \kappa$, i.e.\ the set of all mappings from $\lambda$ into $\kappa$; then, trivially, $|T| = \kappa^\lambda = 2^\kappa$.

Consider the module $N^\prime = N^{(T)}$. For each $\eta\in T$, we denote by $\nu_\eta:N\to N^\prime$ the canonical embedding from $N$ onto the $\eta$th direct summand of $N^\prime$. We also put $K_\alpha = \bigoplus_{\beta<\alpha} N_\beta$ for each $\alpha<\lambda$. In particular, $K_0 = 0$.

Let $G$ be the submodule of $N^\prime$ generated by the set \[\{\nu_\eta(x)-\nu_\zeta(x)\mid \eta,\zeta\in T, (\exists \alpha<\lambda)(x\in K_\alpha\;\&\; \eta\restriction\alpha = \zeta\restriction\alpha)\}.\]
We denote by $\pi:N^\prime\to N^\prime/G$ the canonical projection and, for each $S\subseteq T$, by $\pi_S:N^{(S)}\to N^{(S)}+G/G$ the corresponding restriction of $\pi$. We also put $L = N^\prime/G$ and $L_S = N^{(S)}+G/G$.

\smallskip

First, we claim that $\pi_S$ is a split epimorphism for each $S$ with $|S|\leq\lambda$. This trivially holds for $S = \varnothing$. Otherwise, we use the transfinite induction and the fact that we can well-order the set $S$ by some cardinal number. For the induction step, it then suffices to assume that we have a section $\sigma_{S^\prime}:L_{S^\prime}\to N^{(S^\prime)}$ of $\pi_{S^\prime}$ for some $S^\prime\subset S$ with $|S^\prime|<\lambda$ and we are given an $\eta\in S\setminus S^\prime$. Put $V = S^\prime\cup \{\eta\}$. We want to extend $\sigma_{S^\prime}$ to a section $\sigma_V$ of $\pi_V$. The regularity of $\lambda$ implies that \[\gamma := \sup\{\alpha<\lambda\mid (\exists \zeta\in S^\prime)\,\eta\restriction \alpha = \zeta\restriction\alpha\}<\lambda.\] Since $K$ is irredundant, there is a submodule $C_\gamma\subseteq N$ such that $K_\gamma\oplus C_\gamma = N$. It follows from the definitions of $G$ and $\gamma$ that $\pi\nu_\eta(K_\gamma)\subseteq L_{S^\prime}$ and also that $L_{S^\prime}\cap\pi\nu_\eta(C_\gamma) = 0$ whence $L_V = L_{S^\prime}\oplus\pi\nu_\eta(C_\gamma)$. Finally, since $\pi_{\{\eta\}}$ is one-one, we extend $\sigma_{S^\prime}$ to $\sigma_V$ by sending the direct summand $\pi\nu_\eta(C_\gamma)$ into $N^{(V)}$ via $\pi_{\{\eta\}}^{-1}$.

\smallskip

Second, we claim that $\pi$ is not a split epimorphism. To show this, we first pick for each $u\in{}^{<\lambda}\kappa = \{h:\alpha\to \kappa\mid \alpha<\lambda\}$ some $\eta_u\in T$ which extends~$u$. Since $\kappa^{<\lambda} = \kappa\geq |R|+|K|$, we see that the cardinality of the submodule $F = \sum_{u\in\,{}^{<\lambda}\kappa}\nu_{\eta_u}(K_{\Dom(u)})$ of $N^\prime$ is at most $\kappa$. By the definitions of $F$ and $G$, the canonical projection $N^\prime\to N^\prime/(F+G)$ coincides with the coproduct map $f^{(T)}:N^\prime \to C^{(T)}$. So $f^{(T)} = g\pi$ for the natural projection $g:L\to C^{(T)}$. Put $D = \Ker(g) = F + G/G$.

As in the proof of \cite[Lemma~4.1]{Sa1}, we use the following commutative diagram with exact rows and columns:

\[\begin{CD}
		0				@>>> \Hom_R(C^{(T)},N) @>{\Hom_R(g,N)}>> \Hom_R(L,N) \\
		@.			@V{\Hom_R(C^{(T)}, f)}VV		@V{\Hom_R(L, f)}VV 	\\
	  0				@>>> \Hom_R(C^{(T)},C) @>{\Hom_R(g,C)}>> \Hom_R(L,C) \\	
	 @. 			@V{\varepsilon}VV 		@V{\xi}VV		\\
	\Hom_R(D,K)	@>{\delta}>>	\Ext_R^1(C^{(T)},K)  @>{\Ext_R^1(g,K)}>> \Ext_R^1(L,K).
\end{CD}\]

\medskip

Since $|D|, |K|\leq\kappa$, we have $|\Img(\delta)|\leq |\Hom_R(D,K)|\leq \kappa^\kappa = 2^\kappa = |T|$. For each nonempty $X\subseteq T$, we consider the epimorphism $\psi_X:C^{(T)}\to C$ defined as the identity mapping on the summands indexed by the elements from $X$ and as the zero mapping elsewhere. Since $f$ is not a split epimorphism, any two distinct nonempty sets $X,Y\subseteq T$ provide us with distinct nonzero $\varepsilon(\psi_X),\varepsilon(\psi_Y)\in\Ext_R^1(C^{(T)},K)$. It follows that there exists a nonempty $X\subseteq T$ such that $\varepsilon(\psi_X)\notin \Img(\delta)$. Using the commutativity of the lower ractangle in the diagram above and the exactness of rows and columns, we see that not only $\psi_X$ does not factorize through $f$ but $\psi_X g$ does not factorize through~$f$ either. On the other hand, $\psi_X f^{(T)} = \psi_X g\pi$ trivially factorizes through~$f$, which immediately implies that $\pi$ is not a split epimorphism.

\smallskip

Now we can take the least cardinal $\lambda^\prime$ such that there exists $I\subseteq T$ of cardinality~$\lambda^\prime$ and with the property that $\pi_I$ is not a split epimorphism. By the first part, we know that $\lambda<\lambda^\prime$. If $\lambda^\prime>\aleph_n$, we are done: indeed, $N^\prime$ is a direct sum of copies of $N$; the condition (iii) holds since $\pi_S$ splits for each finite $S\subseteq T$; for the same reason, we get that $L$ is a directed union of modules $\Img(\pi_S)\in\Add(M)$ where $S$ runs through the finite subsets of $T$; finally, the condition (ii) is satisfied since $\pi_S$ is split even for each $S\subseteq T$ of cardinality at most $\aleph_n$.

In the remaining case of $\lambda^\prime\leq\aleph_n$, we fix a well ordering of $I = \{\eta_\alpha\mid \alpha<\lambda^\prime\}$. Put $K^\prime = \Ker(\pi_I)$. By the choice of $\lambda^\prime$, we have the splitting filtration $(K^\prime_\alpha\mid \alpha\leq\lambda^\prime)$ of $K^\prime$ where $K^\prime_\alpha = \Ker(\pi_{I_\alpha})$ and $I_\alpha = \{\eta_\beta\mid \beta<\alpha\}$ for each $\alpha\leq\lambda^\prime$. Thus there is a submodule $N^\prime_\alpha$ of $K^\prime_{\alpha+1}$ such that $K^\prime_{\alpha+1} = K^\prime_\alpha\oplus N^\prime_\alpha$ for each $\alpha<\lambda^\prime$. Then $\bigoplus_{\alpha<\lambda^\prime} N^\prime_\alpha = K^\prime$ is a local direct summand in $N^{(I)}$ and hence also in $N^\prime$. We know that $\lambda^\prime$ is regular since $\lambda^\prime<\aleph_\omega$, and so (the latter decomposition of) $K^\prime$ is irredundant by the choice of $\lambda^\prime$. We redefine $N := N^\prime$, $K := K^\prime$, $\lambda:=\lambda^\prime$ and $N_\alpha := N_\alpha^\prime$ for each $\alpha<\lambda^\prime$ and we return to the beginning of the second paragraph of the proof; this time with a larger $\lambda$. It is clear that this reiteration of the whole process is needed at most $n$ times until we guarantee that $\lambda^\prime>\aleph_n$.
\end{proof}


\end{document}